\documentclass[11pt,oneside,reqno]{amsart}
\usepackage[ngerman, english]{babel}
\usepackage{amsfonts, amsmath,amsthm,booktabs,amssymb}
\usepackage{url}
\usepackage{upgreek}
\usepackage{mathrsfs} 

\usepackage{hyperref}
\hypersetup{
    colorlinks = true,%
    linkcolor = blue,%
    citecolor = red,%
    pdfcreator={LaTeX},
    pdftitle={Dimensional Reduction of generalized Seiberg-Witten equations},%
    pdfsubject={generalized seiberg-witten theory},%
    pdfauthor={Varun Thakre, Rukmini Dey},%
    pdfkeywords={gauge theory, generalized dirac operator, non-linear dirac operator, generalized seiberg-witten, geometric-quantization, dimensional reduction of seiberg-witten, hyperkahler manifolds},%
}

\newcommand{\unit}{1\!\!1} 

\theoremstyle{plain}
\newtheorem{theorem}{Theorem}[subsection]
\newtheorem{proposition}{Proposition}[subsection]
\newtheorem{lemma}{Lemma}[subsection]

\theoremstyle{definition}

\newtheorem{example}{Example}[subsection]

\theoremstyle{definition}
\newtheorem{remark}{Remark}[subsection]
\newtheorem{note}{Note}[subsection]

\makeatletter
\def\@tocline#1#2#3#4#5#6#7{\relax
  \ifnum #1>\c@tocdepth 
  \else
    \par \addpenalty\@secpenalty\addvspace{#2}%
    \begingroup \hyphenpenalty\@M
    \@ifempty{#4}{%
      \@tempdima\csname r@tocindent\number#1\endcsname\relax
    }{%
      \@tempdima#4\relax
    }%
    \parindent\z@ \leftskip#3\relax \advance\leftskip\@tempdima\relax
    \rightskip\@pnumwidth plus4em \parfillskip-\@pnumwidth
    #5\leavevmode\hskip-\@tempdima
      \ifcase #1
       \or\or \hskip 1em \or \hskip 2em \else \hskip 3em \fi%
      #6\nobreak\relax
    \dotfill\hbox to\@pnumwidth{\@tocpagenum{#7}}\par
    \nobreak
    \endgroup
  \fi}
\makeatother
 
\begin{document}

\title{Generalized Seiberg-Witten equations on Riemann surface}   
 
\author[R. Dey]{Rukmini Dey}

\address{International Centre for Theoretical Sciences, Bangalore, India}

\email{rukmini@icts.res.in}

\author[V. Thakre]{Varun Thakre}

\address{Department of Mathematics, Harish Chandra Research Institute, 
Allahabad, India}

\email{varunthakre@hri.res.in}

\date{\today} 


\maketitle



\begin{abstract}

In this paper we consider twice-dimensionally reduced, generalized Seiberg-Witten equations, defined on a compact Riemann surface. A novel feature of the reduction technique is that the resulting equations produce an extra ``Higgs field''. Under suitable regularity assumptions, we show that the moduli space of gauge-equivalent classes of solutions to the reduced equations, is a smooth K\"ahler manifold and construct a pre-quantum line bundle over the moduli space of solutions.

\end{abstract}

\tableofcontents
\section{Introduction}
\label{intro}

Dimensional reduction of gauge-theories have been instrumental in the understanding of Topological QFTs (TQFT). As motivating examples, one can consider the vortex equations \cite{hitchin}, which are the dimensional reduction of 4-dimensional Yang-Mills equations, the dimensional reduction and quantization of 3-dimensional Chern-Simons gauge theory. 

Seiberg-Witten gauge theory has been of interest mathematicians, for as a TQFT, it provides new topological invariants which may provide new directions leading towards the classification of smooth, four-dimensional manifolds. Dimensional reduction of Seiberg-Witten equations to two-dimensions has been studied by Martin \& Restuccia \cite{restuccia}, Saclioglua \& Nergiza \cite{nergiz} and Dey \cite{rukmini}. Except for \cite{rukmini}, the reduction does not involve any Higgs field. 

In this paper, we construct a dimensional reduction of \emph{generalized Seiberg-Witten equations}. For dimension three, the generalized Seiberg-Witten equations were introduced by Taubes \cite{taubes} and were extended to dimension four by Pidstrygach \cite{victor}. The central element of this generalization involves construction of a non-linear Dirac operator by replacing the spinor representation $\mathbb{H}$ with a hyperK\"ahler manifold admitting certain symmetries. The reduction technique we use is similar to the one in \cite{hitchin}, \cite{rukmini}. Namely, we first consider the generalized Seiberg-Witten equations on $\mathbb{R}^{4}$ and then project the equations on the complex plane.  The resulting equations are conformally invariant and therefore can be defined on a compact Riemann surface of any genus. 

Under suitable regularity conditions, the moduli space of solutions to the reduced equations is shown to be a smooth K\"ahler manifold. If the K\"ahler 2-form is integral, we show that the Quillen determinant line-bundle on the configuration space, descends as the pre-quantum line bundle over the moduli space of solutions. Regarding the moduli space as the phase space, we define its Hilbert space quantization as the space of holomorpic sections of the Quillen determinant line-bundle as in \cite{D1}. 

Our paper is organized as follows: we first review the requisite preliminaries on the hyperK\"ahler manifolds in section (\ref{preliminaries}) and then proceed to a quick introduction to the non-linear Dirac operator in four dimensions in subsection (\ref{gensw}). Using this, we define the generalized Seiberg-Witten equations. Although the generalization makes sense for any four-dimensional manifold, for the sake of simplicity and with the further exposition in mind, we stick to the simplest case where the base manifold is $\mathbb{R}^{4}$. In the section (\ref{dimensional reduction}), we describe a dimensional reduction technique and define the reduced equations on $\mathbb{R}^{2}$. Using the conformal invariance of the equations, we define them on an arbitrary compact, oriented, Riemann surface.In section (\ref{moduli space}) we show that the moduli space of gauge-equivalent solutions is a smooth, K\"ahler manifold. In the final section (\ref{quantization}), we describe the Quillen determinant line bundle construction on the moduli space.


\section{Definitions and notations}
\label{preliminaries}

A \emph{hyperK\"ahler manifold} $(M, g^{\scriptscriptstyle M}, I_{1},I_{2},I_{3})$ is a $4n$-dimensional Riemannian manifold, endowed with three complex structures satisfying quaternionic relations $I_{1}^{2} = I_{2}^{2} = I_{3}^{2} = I_{1}I_{2}I_{3} =-1$, such that the metric $g^{\scriptscriptstyle M}$ is K\"ahler with respect to each $I_{j}$, $j=1,2,3$.

Infact, for any $\xi_{1},\xi_{2},\xi_{3} \in \mathbb{R}$ such that $\xi_{1}^2 + \xi_{2}^2 + \xi_{3}^2 = 1$, $I_{\xi} := \xi_{1}I_{1} + \xi_{2}I_{2} + \xi_{3}I_{3} \in End(TM)$ is again a K\"ahler structure on $M$. In other words, $M$ carries a family of K\"ahler structures, parametrized by 2-sphere $S^2$. In particular, a hyperK\"ahler manifold is a symplectic manifold in many different ways.

Suppose that a Lie group $G$ acts smoothly on $M$, preserving the hyperK\"ahler structure. Namely, the action is isometric and fixes the 2-sphere of complex structures. Then $G$ preserves the K\"ahler forms $\omega_{1}, \omega_{2}, \omega_{3}$, associated to $I_{1}, I_{2}, I_{3}$ respectively. Additionally, if the three associated symplectic moment maps exist, then they can be combined into a single \emph{hyperK\"ahler moment map} $\mu: M \longrightarrow \mathbb{R}^{3} \otimes \mathfrak{g}^{*}$, where $\mathfrak{g}$ denotes the Lie algebra of $G$. Such an action of $G$ on $M$ for which the hyperK\"ahler moment map exists is said to be \emph{tri-Hamiltonian}.

\begin{example}
Let $M = \mathbb{H}$. Then $T\mathbb{H} = \mathbb{H} \times \mathbb{H}$. For $(h,v) \in T\mathbb{H}$, define the complex structures \[I_{1}(h,v) = (h, -v\mathrm{i}),~~ I_{2}(h,v) = (h, -v\mathrm{j}),~~ I_{3}(h,v) = (h, -v\mathrm{k})\]
We have $\omega = \frac{1}{2}~d\bar{h}\wedge dh$. Consider the $U(1)$-action on $\mathbb{H}$ given by $U(1)\times \mathbb{H} \ni (z,h) \mapsto zh \in \mathbb{H}.$ The action preserves the three K\"ahler structures and is tri-Hamiltonian, with the hyperK\"ahler moment map $\mu: \mathbb{H} \longrightarrow \mathfrak{sp}(1) \cong \mathfrak{sp}(1)^{*}$ given by \[\mu(h) = \frac{1}{2}~\bar{h}\mathrm{i}h.\]
\end{example}


\subsection{Generalized Seiberg-Witten on $\mathbb{R}^{4}$} 
\label{gensw}


Consider the flat Euclidean space $\mathbb{R}^{4} = \mathbb{H}$, with co-ordinates $(x_{0}, x_{1}, x_{2}, x_{3})$. Fix the constant $Spin$-structure $c: \mathbb{H} = T_{x}\mathbb{H} \longrightarrow \mathbb{H}\times \mathbb{H}$, given by:
\begin{equation*}
c(\xi) = \begin{pmatrix}   
                     *       & \gamma(\xi) \\                    
                -\gamma(\xi) &    0 
         \end{pmatrix},
\gamma(\xi) = \begin{pmatrix}   
              \xi_{0} + i\xi_{1} & \xi_{2} + i\xi_{3} \\                    
             -\xi_{2} + i\xi_{3} & \xi_{0} - i\xi_{1}                      
             \end{pmatrix}.
\end{equation*}
Thus, $\gamma(e_{0}) = \unit, ~~ \gamma(e_{j}) = I_{j}$ for $j=1,2,3$. The covariant derivative of a spinor $u:\mathbb{R}^{4} \longrightarrow \mathbb{H}$ is given by:
\[Du(e_{i}) = \displaystyle \frac{\partial u}{\partial x_{i}}.\]

Composing this with Clifford multiplication $c$, we obtain the Dirac operator $\mathcal{D}: C^{\infty}(\mathbb{R}^{4},\mathbb{H}) \longrightarrow C^{\infty}(\mathbb{R}^{4},\mathbb{H})$ on the space of positive spinors
\[\mathcal{D}^{+} = -\frac{\partial }{\partial x_{0}} + i \frac{\partial }{\partial x_{1}} + j \frac{\partial }{\partial x_{2}} + k \frac{\partial }{\partial x_{3}}\] 
We say that a a smooth map $u: \mathbb{R}^{4} \longrightarrow \mathbb{H}$ is \emph{harmonic} if $\mathcal{D}^{+}u = 0$. Clearly, the Dirac operator (and hence also the harmonicity condition) can be easily generalized to the case where $\mathbb{H}$ is replaced by an arbitrary hyperK\"ahler manifold $(M, g^{\scriptscriptstyle M}, I_{1}, I_{2}, I_{3})$.
More precisely, for a hyperK\"ahler manifold $(M, g^{\scriptscriptstyle M}, I_{1}, I_{2}, I_{3})$ and a smooth map $u: \mathbb{R} \longrightarrow M$, 
\[\mathcal{D}u = -\frac{\partial u}{\partial x_{0}} + I_{1} \frac{\partial u}{\partial x_{1}} + I_{2} \frac{\partial u}{\partial x_{2}} + I_{3} \frac{\partial u}{\partial x_{3}}.\]

The second ingredient we need in order to define the generalized Seiberg-Witten equations is a hyperK\"ahler moment map. Assume that $M$ admits a tri-Hamiltonian action of a compact Lie group $G$. Consider $\mathbb{R}^{4}$ with basic co-ordinates $(x_{1}, x_{2}, x_{3}, x_{4})$ and let $P$ denote the trivial principal $G$-bundle over $\mathbb{R}^{4}$. A connection on $P$ is described by a Lie-algebra-valued one-form \[\mathsf{a} = a_{0}~dx_{0} + a_{1}~dx_{1} + a_{2}~dx_{2} + a_{3}~dx_{3}~\in~\Omega^{1}(\mathbb{R}^{4}, \mathfrak{g}),\]
where $a_{i}: \mathbb{R}^{4} \longrightarrow \mathfrak{g}$ are smooth maps. The curvature of $\mathsf{a}$ is a $\mathfrak{g}$-valued 2-form \[ F(\mathsf{a}) = \sum_{i<j} F^{ij}_{\mathsf{a}}~dx_{i}\wedge dx_{j} \in \Omega^{2}(\mathbb{R}^{4}, \mathfrak{g}) \] where, 
\[F^{ij}_{\mathsf{a}} = \Big(\frac{\partial a_{j}}{\partial x_{i}} - \frac{\partial a_{i}}{\partial x_{j}}\Big) + [a_{i}, a_{j}].\]

For a smooth map $u: \mathbb{R}^{4} \longrightarrow M$ and a connection $\mathsf{a}$ on $P$, we define the \emph{twisted} Dirac operator by
\[\mathcal{D}_{\mathsf{a}}u = \left(-\frac{\partial u}{\partial x_{0}} + L^{M}_{u}a_{0} \right)+\sum^{3}_{i=1}I_{i}\left(\frac{\partial u}{\partial x_{i}} + L^{M}_{u}a_{i}\right), \]
where $L^{M}_{u}a_{i}$ denotes the fundamental vector field generated by the infinitesimal action of $G$ on $M$, at a point $u(\cdot)$ given by \[(L^{M}_{u}a_{0})(p) = \left. \frac{d}{dt} ~ \text{exp}(t ~a_{0}(p))\cdot u(p) \right |_{t=0}, ~~~ p \in \mathbb{R}^{4}.\]

The generalized Seiberg-Witten equations for a pair $(u, \mathsf{a})\in C^{\infty}(\mathbb{R}^{4}, M) \times \Omega^{1}(\mathbb{R}^{4}, \mathfrak{g})$ are given by:
\begin{equation}\label{sw in R4}
\left\{  
  \begin{array}{rcl}
     
     F^{+}_{\mathsf{a}} - \mu \circ u & = & 0 \\    
     \mathcal{D}_{\mathsf{a}}u & = & 0

  \end{array}
  \right.
\end{equation}
where, $F^{+}_{\mathsf{a}} \in \Omega^{2}(\mathbb{R}^{4}, \Lambda^{2}_{+}(\mathbb{R}^{4})^{*} \otimes \mathfrak{g})$ is the self-dual part of the curvature $F_{\mathsf{a}}$. In the first equation we use the identification  
$\Lambda^{2}_{+}(\mathbb{R}^{4})^{*} \cong \mathbb{R}^{3}$ and $\mathfrak{g} \cong \mathfrak{g}^{*}$ using an $ad$-invariant metric on $\mathfrak{g}$.

Equivalently, we can write the equations as:
\begin{equation}
\label{sw in R4 disintegrated}
\left\{  
  \begin{array}{rcl}
     F^{01}_{\mathsf{a}} + F^{23}_{\mathsf{a}} & = & \mu_{1}\circ u\\
     F^{02}_{\mathsf{a}} + F^{31}_{\mathsf{a}} & = & \mu_{2}\circ u\\
     F^{03}_{\mathsf{a}} + F^{12}_{\mathsf{a}} & = & \mu_{3}\circ u\\
     \displaystyle  \left(\frac{\partial u}{\partial x_{0}} + L^{M}_{u}a_{0} \right)  & = & \displaystyle {\sum^{3}_{i=1}I_{i}}\left(\frac{\partial u}{\partial x_{i}} + L^{M}_{u}a_{i} \right)\\
  \end{array}
  \right.
\end{equation}

where $\{ \mu_{1},\mu_{2},\mu_{3} \}$ are the moment maps associated to the K\"ahler 2-forms $\omega_{1}, \omega_{2}, \omega_{3}$ respectively.


\section{Dimensional Reduction}
\label{dimensional reduction}

The approach highlighted in \cite{rukmini} for a dimensional reduction of Seiberg-Witten equations can be generalized for the above set up as well. We describe this procedure below.

Henceforth, we set $G=U(1)$. Identify the Lie algebra $\mathrm{i}\mathbb{R} \cong \mathbb{R}$. Assume that the Lie-algebra-valued functions $\{a_{i}\}^{3}_{i=0}$ are independent of $(x_{2}, x_{3})$. Then $a_{0}, a_{1}$ define a connection $\mathsf{a} := a_{0}dx_{0} + a_{1}dx_{1}$ over $\mathbb{R}^{2}$. The maps $a_{2}$ and $a_{3}$, which we re-label as $\phi_{1}$ and $\phi_{2}$, define an auxillary field $\phi = -\left(\phi_{1} + \textrm{i}\phi_{2} \right)$ (also known as \emph{Higgs fields}) on $\mathbb{R}^{2}$. 
The first equation now reads:
\begin{equation}
\begin{aligned}\label{eq2swonR4}
     \ast F_{\mathsf{a}} & = & \mu_{1}\circ u\\
     \left(\frac{\partial \phi_{1}}{\partial x_{1}} - \frac{\partial \phi_{2}}{\partial x_{2}}\right) & = & \mu_{2}\circ u\\
     \left(\frac{\partial \phi_{1}}{\partial x_{2}} + \frac{\partial \phi_{2}}{\partial x_{1}}\right) & = & \mu_{3}\circ u
\end{aligned}
\end{equation}

From a more co-ordinate independent point of view, we have a connection $\mathsf{a}$ on a principal $U(1)$-bundle $P$ over $\mathbb{R}^{2}$ together with an auxiliary field \[\phi \in \Omega^{0}\left(\mathbb{R}^{2}, \mathbb{C}\right).\]

Set $z=x_{0}+\mathrm{i}x_{1}$ and define the a one-form
\begin{align*}
\Phi 
&= \phi dz - \bar{\phi}d\bar{z} \\
&= \Phi^{1,0} - \overline{\Phi^{1,0}}\in \Omega^{1}\left(\mathbb{R}^{2}, \mathbb{C}\right).
\end{align*}
The second and the third equations in \eqref{eq2swonR4} can be combined into a single equation
\begin{equation}
\label{equation for higgs field on R2}
-\ast\overline{\partial} \Phi^{1,0} = \left(\mu_{2}\circ u + \mathrm{i}\mu_{3}\circ u\right) =: \mu_{c} \circ u.
\end{equation}
The equations \eqref{eq2swonR4} now read:
\begin{equation}
\label{eq2swonR4modified}
\left\{
\begin{array}{rcl}
  \ast F_{\mathsf{a}} -\mu_{1}\circ u & = & 0\\ 
   \ast\overline{\partial}\Phi^{1,0} + \mu_{c}\circ u & = & 0
\end{array}
\right.
\end{equation}
The fourth equation in \eqref{sw in R4 disintegrated} can be re-written as
\begin{equation}
\label{simplified equation dirac I}
\left(\frac{\partial u}{\partial x_{0}} + L^{M}_{u}a_{0}\right) - I_{1} \left(\frac{\partial u}{\partial x_{1}} + L^{M}_{u}a_{1} \right) = \left( I_{2}L^{M}_{u}\phi_{1} + I_{3}L^{M}_{u}\phi_{2} \right)
\end{equation}

Let $I_{\mathbb{R}^{2}}$ denote the standard complex structure on $\mathbb{R}^{2}$, given by $\frac{\partial}{\partial x_{1}} = I_{\mathbb{R}^{2}} \left( \frac{\partial}{\partial x_{0}} \right)$ and $\frac{\partial}{\partial x_{0}} = - I_{\mathbb{R}^{2}} \left(\frac{\partial}{\partial x_{1}}\right)$. Observe that
\begin{equation*}
a_{0} = \mathsf{a}\left(\frac{\partial}{\partial x_{0}}\right), ~~ a_{1} = \mathsf{a}\left(\frac{\partial}{\partial x_{1}}\right) = \mathsf{a}\left(I_{\mathbb{R}^{2}} \left(\frac{\partial}{\partial x_{0}}\right)\right)
\end{equation*}
Then the left hand side of \eqref{simplified equation dirac I} can be written as 
\begin{align*}
\left(\frac{\partial u}{\partial x_{0}} + L^{M}_{u}a_{0} \right) - I_{1} \left(\frac{\partial u}{\partial x_{1}} + L^{M}_{u}a_{1} \right) 
&= \left(du \left(\frac{\partial}{\partial x_{0}}\right) + L^{M}_{u}\left(\mathsf{a}\left(\frac{\partial}{\partial x_{0}}\right) \right) \right) \\
&- I_{1} \left(du \left(\frac{\partial}{\partial x_{1}}\right) + L^{M}_{u}\left( \mathsf{a}\left(\frac{\partial}{\partial x_{1}}\right) \right)\right)\\
&= D_{\mathsf{a}} u\left( \frac{\partial}{\partial x_{0}} \right) - I_{1} D_{\mathsf{a}} u\left(I_{\mathbb{R}^{2}} \left(\frac{\partial u}{\partial x_{0}} \right) \right)\\
&= \left(D_{\mathsf{a}}u - I_{1}D_{\mathsf{a}}u \circ I_{\mathbb{R}^{2}} \right) \left( \frac{\partial}{\partial x_{0}} \right) \\
&:= \partial_{\mathsf{a}}u \left( \frac{\partial}{\partial x_{0}} \right)
\end{align*}
On the other hand, observe now that 
\begin{equation*}
\phi_{1} = \Phi\left(\frac{\partial}{\partial x_{0}}\right)~~\text{and}~~ \phi_{2} = \Phi\left(-\frac{\partial}{\partial x_{1}}\right)
\end{equation*}
The right-hand side can be expressed as
\begin{align*}
\left( I_{2} L^{M}_{u} \phi_{1} + I_{3}L^{M}_{u}\phi_{2}\right)
&= \left( I_{2}L^{M}_{u}\left( \Phi\left(\frac{\partial}{\partial x_{0}}\right) \right) + I_{3}L^{M}_{u}\left( \Phi\left(-\frac{\partial}{\partial x_{1}}\right) \right)\right)\\
&= \left( I_{2}L^{M}_{u}\left( \Phi\left(\frac{\partial}{\partial x_{0}}\right) \right) - I_{3}L^{M}_{u}\left( \Phi\left(I_{\mathbb{R}^{2}} \left(\frac{\partial}{\partial x_{0}}\right)\right)\right) \right)\\
&= \left( I_{2}L^{M}_{u}\left(\Phi\left(\frac{\partial}{\partial x_{0}}\right) \right) - I_{1}\left(I_{2}L^{M}_{u}\left( \Phi\left(I_{\mathbb{R}^{2}} \left(\frac{\partial}{\partial x_{0}} \right)\right) \right) \right) \right)\\
&= \left( I_{2}L^{M}_{u}\Phi - I_{1}\left(I_{2}L^{M}_{u}\Phi\circ I_{\mathbb{R}^{2}} \right) \right)\left(\frac{\partial}{\partial x_{0}}\right)\\
&= \left(X_{\Phi}(u)\right)^{1,0}\left(\frac{\partial}{\partial x_{0}}\right) 
\end{align*}

Combining this together with \eqref{eq2swonR4modified}, we get the reduced equations on $\mathbb{R}^{2}$
\begin{equation}\label{dimensionalreductioninR4}
  \left\{
    \begin{array}{lcl}
      \ast F_{\mathsf{a}} -\mu_{1}\circ u  =  0 \\
      \partial_{\mathsf{a}}u - \left(X_{\Phi}(u)\right)^{1,0}= 0 \\      
      \ast\overline{\partial}\Phi^{1,0} + \mu_{c}\circ u  =  0
    \end{array}
  \right.
\end{equation}

The equations are conformally invariant and hence can be defined on manifolds modelled locally on $\mathbb{R}^{2}$, namely, Riemann surfaces.

\subsection{Generalized Seiberg-Witten on Riemann surface}
\label{gen sw on riemann surface}
Let $(\Sigma, g_{\scriptscriptstyle \Sigma}, J_{\scriptstyle \Sigma})$ be a compact, oriented Riemann surface of genus $g$, with a conformal metric $ds^{2} = h^{2}dz\otimes d\bar{z}$. Let $\pi_{P}: P \rightarrow \Sigma$ be a principal $U(1)$-bundle over $\Sigma$. Let $(M, g^{\scriptscriptstyle M}, I_{1}, I_{2}, I_{3})$ be a hyperK\"ahler manifold endowed with a tri-Hamiltonian action of $U(1)$. We denote by $\mathscr{S}^{\infty}:=C^{\infty}(P,M)^{U(1)}$ the space of smooth $U(1)$-equivariant maps $u:P \rightarrow M$. Denote by $\mathcal{A}(P)$ the space of connections on $P$.

For $u\in \mathscr{S}^{\infty}$, we define the covariant derivative of $u$ with respect to $\mathsf{a}$ by \[D_{\mathsf{a}}u = du + L^{M}_{u}\mathsf{a}\in \Omega^{1}(P, u^{*}TM)^{U(1)}_{hor},\]  where the subscript ``\emph{hor}" denotes that the 1-form is horizontal. This therefore descends to a one form on $\Sigma$ with values in $u^{\ast}TM/U(1)$. The complex structure $I_{1}$ determines a $U(1)$-invariant complex structure on $u^{\ast} TM \longrightarrow P$ and hence also on $u^{\ast} TM/U(1) \longrightarrow \Sigma$. We denote by $\partial_{\mathsf{a}}u$ the $(1,0)$-part of the 1-form $D_{\mathsf{a}}u$, with respect to $I_{1}$. Namely,
\[\partial_{\mathsf{a}}u = \frac{1}{2}\left( D_{\mathsf{a}}u - I_{1} \circ D_{\mathsf{a}}u \circ J_{\scriptstyle \Sigma} \right). \]

Define the configuration space \[\mathcal{C}^{\infty} =  \mathcal{A}(P)\times \mathscr{S}^{\infty} \times \Omega^{1}(\Sigma, \mathbb{R}).\] The space $\mathcal{C}^{\infty}$ is an infinite-dimensional Frech\'et manifold with an action of the gauge group $\mathcal{G}^{\infty} = C^{\infty}(P, U(1))$ given by \[g \cdot (\mathsf{a},u,\Phi) \longmapsto \left(~ \mathsf{a}+g^{-1}dg,~g\cdot u~, \Phi\right).\] 
Note that the gauge group does not act on the Higgs field!

For $(\mathsf{a},u,\Phi)\in \mathcal{C}$, we define the dimensional reduction of generalized Seiberg-Witten equations on $\Sigma$ by:
\begin{equation} \label{seiberg-witen on riemann surface}
  \left\{
    \begin{array}{lcl}
      *F_{\mathsf{a}} - \mu_{1}\circ u = 0 \\
      \partial_{\mathsf{a}}u - \left(X_{\Phi}(u)\right)^{1,0}= 0 \\      
      *\overline{\partial}\Phi^{1,0} + \mu_{c} \circ u  =  0
    \end{array}
  \right.
\end{equation}
The first and third equation requires some explanation. For the first equation, we consider $F_{\mathsf{a}} \in \Omega^{2}(P, \mathbb{R})_{hor}$. For the third equation, observe that $\mu_{c} \circ u: P \longrightarrow \mathbb{C}$ is $U(1)$-invariant and therefore descends to a complex-valued map on $\Sigma$, which we again denote by $\mu_{c} \circ u$. 

The equations \eqref{seiberg-witen on riemann surface} are invariant under the action of $\mathcal{G}^{\infty}$.


\section{Moduli space}
\label{moduli space}

In this section, we shall construct the moduli space of gauge equivalent solutions to \eqref{seiberg-witen on riemann surface}. We begin by fixing the Sobolev completion of the configuration space $\mathcal{C}$ so as to use the implicit function formulation on Banach manifolds. In order to ensure smoothnesss of the moduli space, we need to consider suitable perturbation of the equations \eqref{seiberg-witen on riemann surface}.

Throughout the rest of the section we shall assume that the tri-Hamiltonian $U(1)$-action on $M$ is semi-free; i.e, outside the set of fixed points $M^{U(1)}$, the action is free.

\subsection{Sobolev completions}
Let $\pi_{E}: E \longrightarrow \Sigma$ be the fibred product $P\times_{\scriptscriptstyle U(1)}M$. Consider an embedding $\iota: E \hookrightarrow \mathbb{R}^{N}$. We define the $W^{1,p}$-norm of $u \in \Gamma(\Sigma, E)$ to be the sum of $W^{1,p}$-norm of the components of $\iota \circ u$. Define $\mathscr{S}^{1,p}$ to be the completion of $\mathscr{S}^{\infty}\cong \Gamma(\Sigma, E)$ in the $W^{1,p}$-norm. Then $\mathscr{S}^{1,p}$ is a Banach manifold. For $p>2$, we have a compact embedding $W^{1,p} \hookrightarrow C^{0}$. This implies that for $p>2$, every $u \in \mathscr{S}^{1,p}$ is continuous. Note that the completion is independent of the embedding. 

Fix a smooth fiducial connection $A_{0}$ on $P$ and define 
\[\mathcal{A}^{1,p}(P): = A_{0} + \Omega^{1}(\Sigma, \mathbb{R})_{W^{1,p}}.\]

For $p >2$, the Sobolev multiplication theorem $W^{1,p}\otimes W^{1,p} \longrightarrow L^{p}$ implies $\partial_{\mathsf{a}}u \in \Omega^{1,0}\left(\Sigma, E_{u}\right)_{L^{p}}$, where $E_{u}:= u^{*}TM/U(1)$. Also, since $p>2$, the Sobolev composition law holds. This implies, $\mu_{1}\circ u \in W^{1,p}(\Sigma, \mathbb{R})$ and $\mu_{c}\circ u \in W^{1,p}(\Sigma, \mathbb{C})$.

Finally, we consider the completion $\mathcal{G}^{2,p}$ of the gauge group $\mathcal{G}$ in the $W^{2,p}$-norm. Then $\mathcal{G}^{2,p}$ is a Banach Lie group acting smoothly on $\mathcal{A}^{1,p}$, with its Lie algebra being given by $\text{Lie}(\mathcal{G}^{2,p}) = W^{2,p}(\Sigma, \mathbb{R})$.
 
\subsection{Abstract setup}
For $p>2$, consider the infinite dimensional Banach manifold given by 
\[\mathcal{C}^{1,p} =  \mathcal{A}^{1,p}(P) \times W^{1,p}(\Sigma, E) \times \Omega^{1}(\Sigma,\mathbb{R})_{ W^{1,p}}. \]
The tangent to $\mathcal{C}^{1,p}$ at a point $q:= \left(\mathsf{a},u,\Phi\right) \in \mathcal{C}^{1,p}$ is given by:
\[T_{q}\mathcal{C}^{1,p} = \Omega^{1}(\Sigma, \mathbb{R})_{W^{1,p}} \times W^{1,p}(\Sigma, E_{u}) \times \Omega^{1}(\Sigma,\mathbb{R})_{W^{1,p}}.\]
Consider the infinite-dimensional vector bundle $\mathcal{E}^{p} \longrightarrow \mathcal{C}^{1,p}$, with fibre at a point $q \in \mathcal{C}^{1,p}$ being given by
\[ \mathcal{E}^{p}_{q} = \Omega^{0}(\Sigma, \mathbb{R})_{L^{p}} \times \Omega^{1,0}(\Sigma, E_{u})_{L^{p}} \times\Omega^{0}(\Sigma, \mathbb{C})_{ L^{p}}.\]
Observe that the action of the gauge group $\mathcal{G}^{2,p}:= W^{2,p}(P, U(1))$ on $\mathcal{C}^{1,p}$ lifts to an action on $\mathcal{E}^{p}$. Define the equivariant section 
\begin{equation}
    \begin{aligned}
           &  \mathcal{F}:\mathcal{C}^{1,p} \longrightarrow \mathcal{E}^{p} \\
          \mathcal{F}(\mathsf{a}, u, \Phi) &= \left(*F_{\mathsf{a}} - \mu_{1}\circ u, ~ \partial_{\mathsf{a}}u - \left(X_{\Phi}(u)\right)^{1,0},~ \ast\overline{\partial} \Phi^{1,0} + \mu_{c} \circ u \right)
    \end{aligned}
\end{equation}

Then the solutions to \eqref{seiberg-witen on riemann surface} are the zeroes of $\mathcal{F}$.

\subsection{Linearized Operator}
\label{linearized operator}

The linearization of the equations \eqref{seiberg-witen on riemann surface} at a zero $q=(\mathsf{a}, u, \Phi) \in \mathcal{C}^{1,p}$ of $\mathcal{F}$ gives the operator:

\begin{center}
$D_{q}:T_{q}\mathcal{C}^{1,p} \longrightarrow \mathcal{E}^{p}_{q}$ \\[0.2cm]
$D_{q}\left(\begin{array}{c}\alpha \\ \xi \\  \eta^{1,0} \end{array} \right) \longmapsto 
\left( \begin{array}{c} *d\alpha - d\mu_{1}(\xi) \\ D_{\mathsf{a},u, \Phi}\xi + (L_{u}\alpha)^{1,0} - (X_{\eta})^{1,0} \\ \ast\overline{\partial} \eta^{1,0} + d\mu_{c}(\xi)
\end{array}\right)$
\end{center}
Here $D_{\mathsf{a},\Phi}\xi = (\nabla^{\mathsf{a}}\xi)^{1,0} + \big(\nabla_{\xi}X_{\Phi}\big)^{1,0}$, where $\nabla$ is the Levi-Civita connection  on $(M,g^{\scriptscriptstyle M})$. The induced connection on $u^{*}TM$ is given by
$\nabla^{\mathsf{a}}\xi + \nabla_{\xi}X_{\Phi}(u)$, where $\nabla^{\mathsf{a}}\xi = \nabla \xi + \nabla_{\xi}K^{M}_{\mathsf{a}}$.

The equivariance of the section $\mathcal{F}: \mathcal{C}^{1,p} \longrightarrow \mathcal{E}^{p}$ under the action of the gauge group $\mathcal{G}^{2,p}$ implies that we have the following complex 
\begin{equation}\label{total complex}
0 \rightarrow  W^{2,p}(\Sigma,\mathbb{R})\xrightarrow{d_{1}} T_{q}\mathcal{C}^{1,p}\xrightarrow{d_{2}} \mathcal{E}^{p}_{q} \rightarrow 0
\end{equation}
where $d_{1}(\alpha) = (L_{u}\alpha, ~ d\alpha,~  0)\in T_{q}\mathcal{C}^{k,2}$ and $d_{2} \left(\alpha, \xi, \eta^{1,0}\right) = D_{q}\left(\alpha, \xi, \eta^{1,0}\right)$. Note that if we deform the complex by a homotopy, so as to get rid of the zeroeth order terms, the Euler characteristic or the symbols of the operators remain unchanged. In other words, the complex \eqref{total complex} can be written as a sum of three complexes:
\begin{gather}
0 \rightarrow W^{2,p}(\Sigma, \mathbb{R}) \xrightarrow{d} W^{1,p}(\Sigma, \Lambda^{1}\Sigma) \xrightarrow{d_{\mathsf{a}}} L^{p}(\Sigma, \Lambda^{2}\Sigma) \rightarrow 0 \label{complex 1} \\ 
0 \rightarrow W^{1,p}(\Sigma, E_{u}) \xrightarrow{D_{q}} L^{p}(\Sigma, \Lambda^{0,1}\Sigma \otimes E_{u}) \rightarrow 0 \label{complex 2} \\ 
0 \rightarrow W^{1,p}(\Sigma, \Lambda^{1}\Sigma\otimes \mathbb{C}) \xrightarrow{\ast \bar{\partial}} L^{p}(\Sigma, \mathbb{C}) \rightarrow 0 \label{complex 3}
\end{gather}
Clearly, each of the above complexes is elliptic and consequently, \eqref{total complex} is an elliptic complex.

Let $\delta$ denote the equivariant map $\delta: P \longrightarrow EU(1)$ which is a lift of the classifying map $\tilde{\delta}:\Sigma \longrightarrow BU(1)$. Then $(u,\delta): P \longrightarrow M\times EU(1)$ descends to a map \[\bar{u}: \Sigma \longrightarrow M_{U(1)}:= M\times_{U(1)}EU(1) .\] Define $[u] \in H_{2}(M_{G}, \mathbb{Z})$ to be the push-forward of the fundamental class of $[\Sigma]$ under the map $\bar{u}$.

\begin{proposition}
The operator $d_{1}^{*}+d_{2}: T_{q}\mathcal{C}^{1,p} \longrightarrow \Omega^{0}(\Sigma,\mathbb{R})_{L^{p}} \oplus \mathcal{E}^{p}_{q}$ is a Fredholm operator for every solution $(\mathsf{a},u,\Phi) \in \mathcal{C}^{1,p}$ of \eqref{seiberg-witen on riemann surface} and has a real index given by 
\begin{equation}
\text{Index}~(d_{1}^{*}+d_{2}) = (2n-1)\chi(\Sigma) + 2 \left\langle c^{U(1)}_{1}(TM), [u] \right\rangle + 2g
\end{equation}
where $c^{U(1)}_{1}(TM)$ is the equivariant first Chern class of $TM$. 
\end{proposition}
\begin{proof}
The ellipticity of the complex \eqref{total complex} has the consequence that the operator $d_{1}^{*}+d_{2}: T_{q}\mathcal{C}^{1,p} \longrightarrow \Omega^{0}(\Sigma,\mathbb{R})_{L^{p}} \oplus \mathcal{E}^{p}_{q}$ is Fredholm and therefore has a well-defined index. Since the complex \eqref{total complex} decomposes into three complexes, the index of \eqref{total complex} is the sum of indices of the complexes \eqref{complex 1}, \eqref{complex 2}, \eqref{complex 3}.
The index for the operator 
\begin{equation*}
\Omega^{1}(\Sigma, \mathbb{R}) \longrightarrow \Omega^{0}(\Sigma, \mathbb{R})\oplus \Omega^{0}(\Sigma, \mathbb{R}):~~ \alpha \mapsto \left( d^{*}\alpha, \ast d\alpha \right)
\end{equation*}
is given by $-\chi(\Sigma)$. By Riemann-Roch theorem, the index of the \eqref{complex 2} is given by $2 \langle c^{U(1)}_{1}(TM), [u] \rangle + 2n\chi(\Sigma)$.
Finally, the index for the third complex \eqref{complex 3} is $2g$. It is a simple observation now that $d^{*}_{1} + d_{2}$ is a compact perturbation of these operators. Therefore, the index of $d^{*}_{1} + d_{2}$ is given by:
\begin{equation}
\text{Index}~(d_{1}^{*}+d_{2}) = (2n-1)\chi(\Sigma) + 2 \left\langle c^{U(1)}_{1}(TM), [u] \right\rangle + 2g
\end{equation}
The statement follows. 
\end{proof}

\subsection{Transversality}
In order to prove that the moduli space of gauge equivalent solutions is a smooth Banach manifold, we need to establish transversality. Namely, we need to prove that $\mathcal{F}$ is transverse to the zero section. The techniques used in this section are almost verbatim to the ones in \cite{riera}, used to prove transversality for symplectic vortex equations. 

Define the space of perturbations \[\mathcal{P}:=\left\lbrace (\sigma_{1}, \sigma_{2}, \sigma_{3}) \in \mathcal{C}~ \left \vert \right. g \cdot (\sigma_{1}, \sigma_{2}, \sigma_{3}) = (\sigma_{1}, \sigma_{2}, \sigma_{3}), g \in \mathcal{G} \right\rbrace.\] 

For $c \in \mathbb{R}$, consider the perturbed equations
\begin{equation}\label{perturbed reduced equations}
  \left\{
    \begin{array}{lcl}
      \ast F_{\mathsf{a}} - \mu_{1}\circ u  =  c + \sigma_{1} \\
      \partial_{\mathsf{a}}u - \left(X_{\Phi}(u)\right)^{1,0}=  \sigma_{2} \\      
      \ast\overline{\partial}\Phi^{1,0} + \mu_{c}\circ u  =  \sigma_{3}
    \end{array}
  \right.
\end{equation}

Fix a cohomology class $B \in H^{2}(M_{U(1)}, \mathbb{Z})$ and for a fixed $\sigma \in \mathcal{P}$, define the solution space 
\begin{equation*}
N_{\sigma}(B,c):= \left\lbrace \left(\mathsf{a}, u, \Phi\right) \in \left. \mathcal{C}^{1,p} ~ \right |  ~ [u] = B ~~ \text{and} ~~ \left(\mathsf{a}, u, \Phi\right) ~~\text{satisfy}~~ \eqref{perturbed reduced equations} \right\rbrace
\end{equation*}
Observe that $N_{\sigma}(B,c)$ is invariant under the action of the gauge group $\mathcal{G}^{2,p}$. Let $M^{U(1)}$ denote the fixed points of the $U(1)$-action on the hyperK\"ahler manifold $M$.
Define
\begin{equation*}
C_{0} = \mu_{1}\left( M^{U(1)} \right) - 2\pi\frac{\text{deg}(P)}{\text{Vol}(\Sigma)} \subset \mathbb{R}
\end{equation*}
Define $\mathscr{S}^{1,p}_{*}:=\left\lbrace u \in \left. \mathscr{S}^{1,p}~ \right | [u] = B,~ u\left(P\right)\nsubseteq M^{U(1)} \right\rbrace$.

\begin{lemma}[\cite{riera}, Lemma 3.4.1, Cor. 3.4.2] \label{freeness of gauge action}
\begin{enumerate}
\item Let $c \in \mathbb{R}\setminus C_{0}$ and define $\mathcal{P}_{c} = \left\lbrace \sigma \in \left. \mathcal{P}~ \right. | \left. ~ \right|\sigma_{1}| < d(c, C_{0}) \right\rbrace$. Then, for $\sigma \in \mathcal{P}_{c}$, if $(\mathsf{a}, u, \Phi)$ satisfy \eqref{perturbed reduced equations}, then $u\left(P\right)\nsubseteq M^{U(1)}$.

\item If $\sigma \in \mathcal{P}_{c}$, the action of $\mathcal{G}^{2,p}$ on $N_{\sigma}(B,c)$ is free.
\end{enumerate}

\end{lemma}

\begin{proof}
\begin{enumerate}
 \item Let $u(P) \subset M^{U(1)}$. This implies $u(\Sigma) \subset \Sigma \times M^{U(1)}$. Since $\Sigma$ is connected, the $\mu_{1}\circ u$ assumes a constant value on $\Sigma$. Integrating the first equation in \eqref{perturbed reduced equations} and using Chern-Weil theory, we obtain that $|\sigma_{1}|_{C^{0}} > d(c, C_{0})$.

\item Observe that if $(\mathsf{a}, u, \Phi)$ have non-trivial stabilizers, then $u(P) \in M^{(1)}$.
\end{enumerate}
The statements follow.
\end{proof}

Define the moduli space of solutions to be the quotient
\begin{equation}
\mathcal{M}_{\sigma}(B,c):= N_{\sigma}(B,c)/\mathcal{G}^{2,p}
\end{equation}

\begin{theorem}[\cite{riera}, Theorem 3.4.4]
Let $c \in \mathbb{R}\setminus C_{0}$. Then for any $\sigma \in \mathcal{P}_{c}$, the moduli space $\mathcal{M}_{\sigma}(B,c)$ is a smooth manifold of real dimension 
\begin{equation*}
(2n-1)\chi(\Sigma) + 2 \left\langle c^{U(1)}_{1}(TM), [u] \right\rangle + 2g
\end{equation*}
\end{theorem}

\begin{proof}

Let $m>0$ be a large positive integer and $\mathcal{P}^{m}_{c}$ denote the completion of $\mathcal{P}_{c}$ with respect to the $C^{m}$-norm. Define
$\mathcal{C}^{1,p}_{\ast} = \mathcal{A}^{1,p} \times \mathscr{S}^{1,p}_{\ast} \times \Omega^{1}(\Sigma, \mathbb{C})W^{1,p}$ and let $\mathcal{B}^{m} :=\mathcal{C}^{1,p}_{\ast} \times \mathcal{P}^{m}_{c}$. Then, we have a Banach vector bundle $\mathcal{W} \longrightarrow \mathcal{B}^{m}$, whose fibre at a point $(\mathsf{a}, u, \Phi, \sigma)$ is $\Omega^{1}(\Sigma, \mathbb{R})_{L^{p}} \oplus \Omega^{1,0}(\Sigma, E_{u})_{L^{p}} \oplus \Omega^{0}(\Sigma, \mathbb{C})_{L^{p}}$ and a section of the bundle $\mathcal{F}: \mathcal{B}^{m} \longrightarrow \mathcal{W}$ given by 
\begin{multline}
\mathcal{F}(\mathsf{a}, u, \Phi, \sigma) = \left(\ast F_{\mathsf{a}} - \mu_{1}\circ u  -  c + \sigma_{1},~ \partial_{\mathsf{a}}u - \left(X_{\Phi}(u)\right)^{1,0} - \sigma_{2},~ \ast\overline{\partial}\Phi^{1,0} + \mu_{c}\circ u  -  \sigma_{3}\right)
\end{multline}
By definition, $\mathcal{F}^{-1}(0) = N_{\sigma}(B,c)$. To prove that $N_{\sigma}(B,c)$ is smooth, we need to prove that $\mathcal{F}$ is transverse to the zero section.

Observe that, modulo the infinitesimal action of the gauge group $\mathcal{G}^{2,p}$, $D\mathcal{F}$ is an elliptic operator and hence its image is closed. We claim that the image is also exhaustive. For if not, then there would exist a non-trivial complement $\mathcal{E}^{l}_{q}$, where $\frac{1}{p} + \frac{1}{l} = 1$ and hence a non-zero element $(\alpha_1, \xi_1, \eta_1) \in \mathcal{E}^{l}_{q}$, such that for any $(\alpha, \xi,\eta,\sigma) \in \mathcal{B}^{m}$, 
\[\int_{\Sigma} \langle *d\alpha - d\mu_{1}(\xi) + \sigma_1, \alpha_1  \rangle = 0,\]
\[\int_{\Sigma} \langle D_{\mathsf{a},u, \Phi}\xi + (L_{u}\alpha)^{1,0} - (X_{\eta})^{1,0} + \sigma_2, \xi_1  \rangle = 0,\] and
\[\int_{\Sigma} \langle \ast\overline{\partial} \eta^{1,0} + d\mu_{c}(\xi) + \sigma_3, \eta_1  \rangle = 0 \]

Let $\Sigma_0 = \{ x \in \Sigma ~ | ~ u(x) \in \Sigma \times M^{U(1)} \}$. Then, for a section $\xi \in Hom(TP, u^{*}TM)_{hor} \cong Hom(T\Sigma, u^{*}TE)$ with support in $\Sigma \setminus \Sigma_0$, \\ $\xi \in Hom(T\Sigma, u^{*}T\mathcal{E}) $. This implies $\xi_1 = 0$ on $\Sigma \setminus \Sigma_0$ and hence on $\Sigma$. For if not, we can choose $\xi = \alpha = \eta = 0$ and $\xi_1$ to be a suitable bump function that makes the integral non-zero. By similar arguments, $\alpha_1 = 0$ and $\eta_1 = 0$. Hence $\mathcal{N}_{\sigma}(B,c)$ is a smooth Banach manifold. By Lemma \eqref{freeness of gauge action}, the action of $\mathcal{G}^{2,p}$ on $\mathcal{N}_{\sigma}(B,c)$ is free. Uhlenbeck's gauge fixing theorem guarantees existence of local slices for the action of the gauge group. This implies that the quotient $\mathcal{M}_{\sigma}(B,c)$ is a smooth, Banach manifold.
\end{proof}

\subsection{K\"ahler structure on Moduli space}

In this section we show that the moduli space can be realized as a Marsden-Weinstein or a symplectic quotient of a submanifold of the configuration space $\mathcal{C}^{1,p}$. The arguments in this section follow the work in \cite{becker} on moduli space of Seiberg-Witten equations on K\"ahler surfaces.

Recall that for a finite dimensional symplectic manifold $(M,\omega)$, with a Hamiltonian action of a Lie group $G$, there exists a moment map $\mu: M \longrightarrow \mathfrak{g}^{*}$, which is unique upto addition by constants in the centre of $\mathfrak{g}$. The equivariance of the moment map imples that the zero-locus of $\mu$ is $G$-invariant. If $0$ is a regular point of $\mu$, then $\mu^{-1}(0)/G$ is again a symplectic manifold. This follows from the well-known Marsden-Weinstein reduction theorem. If $M$ is K\"ahler and the group action preserves both the metric and the symplectic form, then the quotient is a K\"ahler manifold. If $M$ is a hyperK\"ahler manifold and the group action preserves the metric and all the three K\"ahler forms, then the quotient is again a hyperK\"ahler manifold. 

We shall now turn to the infinite dimensional analogue of the above quotient constructions, namely for an action of the gauge group $\mathcal{G}^{2,p}$ on the configuration space $\mathcal{C}^{1,p}$. The configuration space $\mathcal{C}^{1,p}$ carries a hypercomplex structure, i.e, a triple $\mathcal{I}_{1}, \mathcal{I}_{2}, \mathcal{I}_{3}$ of almost complex structures, that obeys quaternionic relations:
\begin{equation*}
\mathcal{I}_{1} = \begin{pmatrix}   
                * & 0 & 0 \\                    
               0 & -I & 0\\                      
               0 & 0 & -*                       
              \end{pmatrix},~~~                 
\mathcal{I}_{2} = \begin{pmatrix}
               0 & 0 & * \\
               0 & -J & 0 \\
               * & 0 & 0
              \end{pmatrix},~~~
\mathcal{I}_{3} = \begin{pmatrix}
               0 & 0 & -1 \\
               0 & K & 0 \\
               1 & 0 & 0
              \end{pmatrix}
\end{equation*}

We now prove that the Nijenhuis tensor for the each $\mathcal{I}_{i}, ~ i=1,2,3$ vanishes. The Nigenhuis tensor for the $i^{th}$ almost complex structure is given by:
\[\mathscr{N}^{i}_{q}(X,Y) = [X, Y] + \mathcal{I}^{u}_{i}[X, \mathcal{I}^{u}_{i} Y] + \mathcal{I}^{u}_{i}[\mathcal{I}^{u}_{i} X, Y] - [\mathcal{I}^{u}_{i}X, \mathcal{I}^{u}_{i}Y].\]
The expression splits into three summations: \[\mathscr{N}^{i}_{q}(X,Y) = N^{i}_{a}(\alpha_{1}, \alpha_{2}) \oplus N^{i}_{u}(\xi_{1},\xi_{2}) \oplus N^{i}_{\Phi}(\eta_{1}, \eta_{2}).\] 
The summation corresponding to the connection component vanishes, as the structure group is Abelian. The one corresponding to the Higgs field vanishes as $\ast$ is independent of the base point. It only remains to check that the Nijenhuis tensor vanishes for the spinor component
\[N^{i}_{u}(\xi_{1},\xi_{2}) = [\xi_{1}, \xi_{2}] + I^{u}_{i}[\xi_{1}, I^{u}_{i} \xi_{2}] + I^{u}_{i}[I^{u}_{i} \xi_{1}, \xi_{2}] - [I^{u}_{i}\xi_{1}, I^{u}_{i}\xi_{2}].\]

But since $I_{i}, ~ i=1,2,3$ are integrable, the Nijenhuis tensor corresponding to $I_{i}$, $N^{i}_{M}(\cdot, \cdot) \equiv 0$ and therefore $N^{i}_{u}(\xi_{1},\xi_{2}) = N^{i}_{M}(\cdot, \cdot)|_{u} = 0$. This holds for all $i = 1,2,3$. Therefore $\mathcal{I}_{i}$, $i=1,2,3$ are integrable and hence $\mathcal{I}_{1}, \mathcal{I}_{2}, \mathcal{I}_{3}$ define a hyperK\"ahler structure on the configuration space.

The $L^{2}$-metric on $\mathcal{C}^{1,p}$, defined by 
\begin{equation*}
g^{\scriptscriptstyle\mathcal{C}}(X,Y) = \frac{1}{2}\int_{\Sigma} *\alpha_{1}\wedge \alpha_{2} +  \frac{1}{2}\int_{\Sigma} g^{\scriptscriptstyle M}_{u}(\xi_{1}, \xi_{2})~\omega_{\Sigma} + \frac{1}{2}\int_{\Sigma} *\eta_{1}\wedge \eta_{2} 
\end{equation*}
where, $X=(\alpha_{1},\xi_{1},\eta_{1})$, $Y=(\alpha_{2},\xi_{2},\eta_{2}) \in T_{q}\mathcal{C}^{1,p}$. Here, the pull-back metric $ g^{\scriptscriptstyle M}_{u}: u^{\ast}TM \otimes u^{\ast}TM \longrightarrow \mathbb{R}$ is defined by
\[g^{\scriptscriptstyle M}_{u}((p,v), (p,w)) = g^{\scriptscriptstyle M}_{u(p)} (v,w), ~~ (p,v), (p,w) \in u^{\ast}TM \subset P \times TM.\]
The metric $g^{\scriptscriptstyle\mathcal{C}}$ is $\mathcal{G}^{2,p}$-invariant and is Hermitian with respect to all three complex structures; i.e $g^{\scriptscriptstyle\mathcal{C}}(\mathcal{I}_{i}X, \mathcal{I}_{i}Y) = g^{\scriptscriptstyle\mathcal{C}}(X,Y)$ for all $i = 1,2,3$. Therefore, the associated 2-forms $\Omega_{i}(\cdot, \cdot) = g^{\scriptscriptstyle\mathcal{C}}(\mathcal{I}_{i} (\cdot),\cdot )$ are non-degenerate, closed and hence K\"ahler \cite{werner}. They are also preserved by the action of the gauge group. Thus the natural $L^{2}$-metric on the configuration space $\mathcal{C}^{1,p}$ is a hyperK\"ahler metric. 

\begin{note}
Although $u^{\ast}g^{\scriptscriptstyle M}$ denotes the pull-back of a metric, the definition of the pull-back metric here does not involve differential of $u$, unlike the pull-back of differential forms.
\end{note}

\begin{remark}
The 2-forms $\Omega_{i}(\cdot, \cdot)$ define a \emph{weak hyperK\"ahler structure} on the configuration space. Namely, each closed form $\Omega_{i}(\cdot, \cdot): T\mathcal{C}^{1,p} \longrightarrow T^{\ast}\mathcal{C}^{1,p}$ is injective, but not surjective.  
\end{remark}

The action of the gauge group $\mathcal{G}^{2,p}$ on $\mathcal{C}^{1,p}$ preserves both the metric and the hyperK\"ahler structure. The first and the third equations of \eqref{seiberg-witen on riemann surface} can be interpreted as moment maps for the hyperK\"ahler action of the gauge group, corresponding to the symplectic 2-forms $\Omega_{1}$ and $\Omega_{c}:= \Omega_{2}+\mathrm{i}\Omega_{3}$ respectively. This can indeed be seen as follows:

The fundamental vector field for the infinitesimal action of the gauge group $\mathcal{G}^{2,p}$, at a point $q\in \mathcal{C}^{1,p}$ is given by $L^{\mathcal{C}}_{q} \gamma = (d\gamma, ~ L_{u}\gamma,~ 0)$. Define 
\begin{equation} \label{moment map on config space mu 1}
\tilde{\mu}_{\scriptscriptstyle\mathcal{I}_{1}}: \mathcal{C}^{1,p}_{\ast} \longrightarrow \text{Lie}(\mathcal{G}^{2,p})^{*}, ~~~~~~
\langle \tilde{\mu}_{\scriptscriptstyle\mathcal{I}_{1}}, \gamma \rangle (q) = \frac{1}{2}\int_{\Sigma} \left(\ast F_{\mathsf{a}} \cdot \gamma - \langle \gamma, \mu_{1}\circ u\rangle \right)\omega_{\scriptscriptstyle\Sigma}
\end{equation}
where $\langle \cdot, \cdot \rangle$ denotes the pairing. Therefore, for $X=(\alpha, \xi, \eta) \in T_{q}\mathcal{C}^{1,p}$,
\begin{align*}
\langle d\tilde{\mu}_{\scriptscriptstyle\mathcal{I}_{1}}(X), \gamma \rangle (q)
&= \frac{1}{2}\int_{\Sigma} \left(\ast d\alpha \cdot \gamma - \left\langle \gamma, d(\mu_{1}\circ u)(\xi)\right\rangle \right)\omega_{\scriptscriptstyle\Sigma}\\
&= \frac{1}{2}\int_{\Sigma} \left(- \ast d\gamma\wedge\alpha - g^{M}_{u}(I_{1}L^{M}_{u}\gamma, \xi) \right)\omega_{\scriptscriptstyle\Sigma} \\
&= \iota_{L^{\mathcal{C}}_{q} \gamma} \Omega_{1}(X)
\end{align*}

Let $\gamma_{1},~\gamma_{2} \in \text{Lie}(\mathcal{G}^{2,p})$ and define $\gamma = \gamma_{1} + \mathrm{i}\gamma_{2}$. Define 
\begin{equation*}
 \tilde{\mu}_{c}:\mathcal{C}^{1,p} \longrightarrow \mathbb{C}^{*}, ~~~~~~~ \langle \tilde{\mu}_{c}, \gamma \rangle (q) = \frac{1}{2}\int_{\Sigma} \left(\gamma \cdot \ast\overline{\partial}\Phi^{1,0} - \langle \gamma, (\mu_{c} \circ u)\rangle \right)\omega_{\scriptscriptstyle \Sigma} 
\end{equation*}
Therefore, again, for $X=(\alpha, \xi, \eta) \in T_{q}\mathcal{C}^{1,p}$,
\begin{align*}
\langle d\tilde{\mu}_{c}(X), \gamma \rangle (q)
&= \frac{1}{2} \int_{\Sigma} \left(\ast \overline{\partial}\eta^{1,0} \cdot \gamma - \left\langle \gamma, d(\mu_{c}\circ u)(\xi)\right\rangle \right)\omega_{\scriptscriptstyle\Sigma}\\
&= \frac{1}{2}\int_{\Sigma}\left( -\ast\overline{\partial}\gamma\wedge\eta^{1,0} - \left(g^{M}_{u}(I_{2}L^{M}_{u}\gamma, \xi) + g^{M}_{u}(I_{3}L^{M}_{u}\gamma, \xi)\right) \right)\omega_{\scriptscriptstyle\Sigma} \\
&= \iota_{L^{\mathcal{C}}_{q} \gamma} \big(\Omega_{2}(X) + \mathrm{i}\Omega_{3}(X) \big) \\
&= \iota_{L^{\mathcal{C}}_{q}\gamma} \Omega_{c}(X)
\end{align*}

Combining this into a single moment map $\tilde{\mu}: \mathcal{C}^{1,p} \longrightarrow \mathbb{R}^{3}\otimes\text{Lie}(\mathcal{G}^{2,p})^{*}$, we get the hyperK\"ahler moment map for the action of the gauge group. In order to show that the moduli space of solutions to \eqref{seiberg-witen on riemann surface} is a symplectic quotient, we need to impose the second equation. 
Denote by $\mathcal{C}' \subset \mathcal{C}^{1,p}$, the subspace of all the solutions to second and third equations in \eqref{seiberg-witen on riemann surface} 
\begin{equation}
\mathcal{C}' = \left\{\left.\left(\mathsf{a},u,\Phi \in \mathcal{C}^{1,p}_{\ast} \right)~~ \right\vert ~~ \left(\begin{array}{c}\partial_{\mathsf{a}}u - (X_{\Phi})^{1,0} \\ \ast\overline{\partial}\eta^{1,0} - \mu_{c}\circ u\end{array}\right) = 0 \right\}
\end{equation}

\begin{lemma}
$\mathcal{C}'$ is a complex submanifold of $\mathcal{C}^{1,p}$ with respect to the complex structure $\mathcal{I}_{1}$.
\end{lemma}
\begin{proof}
We only need to show that $T\mathcal{C}'\subset T\mathcal{C}^{1,p}$ is a complex sub-bundle. Abbreviate $D_{\mathsf{a}, u, \Phi}:= D$ for simplicity. The tangent space at a point $q:=(\mathsf{a}, u, \Phi)\in\mathcal{C}'$ is given by:
\begin{equation}
T_{q}\mathcal{C}' = \left\{\left.\left(\begin{array}{c}\xi \\ \alpha \\ \eta \end{array} \right)~~\right\vert~~ \left(\begin{array}{c} D\xi + (L_{u}\alpha)^{1,0} - (X_{\eta})^{1,0}\\ \ast\overline{\partial}\eta^{1,0} - d\mu_{c}(\xi)\end{array}\right) = 0 \right\}
\end{equation}
To see that $\mathcal{I}_{1}$ preserves $T_{q}\mathcal{C}'$, observe that
\[\mathcal{I}_{1} \left(\begin{array}{c}\alpha \\ \xi \\ \eta \end{array} \right) = \left(\begin{array}{c} *\alpha \\ -I_{1}\xi \\ -*\eta^{1,0} \end{array} \right). \] But since $D$ is a Cauchy-Riemann operator and $(L_{u}(*\alpha))^{1,0} = -I_{1}(L_{u}\alpha)^{1,0}$ and similarly $\left(X_{*\eta}\right)^{1,0} = -I_{1}\left(X_{\eta}\right)^{1,0}$, we get 
\[D(-I_{1})\xi + (L_{u}(*\alpha))^{1,0} - (X_{*\eta})^{1,0} = (-I_{1})\left(D\xi + (L_{u}\alpha)^{1,0} - (X_{\eta})^{1,0}\right) = 0 \] 
Since, $\eta^{1,0}\in\Omega^{1,0}(\Sigma,\mathbb{C})$, $*\eta^{1,0} = -\mathrm{i}\eta$. Also, $d\mu_{c}(-I_{1}\xi) = -\mathrm{i}d\mu_{c}(\xi)$. Clearly then, \[*\ast\overline{\partial} \ast\eta^{1,0} - d\mu_{c}(I_{1}\xi) = -\mathrm{i}\left(\ast\overline{\partial} \eta^{1,0} - d\mu_{c}(\xi)\right) = 0 .\] 

Therefore $\mathcal{I}_{1}$ preserves $T\mathcal{C}'\subset T\mathcal{C}^{1,p}$.
\end{proof}

The $L^{2}$-metric restricts to a K\"ahler metric on $\mathcal{C}'$. The induced action of the gauge group $\mathcal{G}^{2,p}$ preserves the induced metric and the symplectic two-form $\Omega_{1}$.
$\mathcal{C}'$ also admits a momentum map $\mu'_{\scriptscriptstyle\mathcal{I}_{1}}$ which is just a restriction of the momentum map $\tilde{\mu}_{\scriptscriptstyle\mathcal{I}_{1}}$ on the configuration space.
We denote this restriction by $\tilde{\mu}_{\scriptscriptstyle\mathcal{I}_{1}}$ itself. The solutions to the dimensionally reduced generalized Seiberg-Witten equations now correspond to the quotient of the zero locus
of the momentum map $\tilde{\mu}_{\scriptscriptstyle\mathcal{I}_{1}}$ by the gauge group $\mathcal{G}^{2,p}$. Using the standard arguments in K\"ahler geometry, we now show that the $L^{2}$-metric that is induced on the quotient $\mathcal{M} := \tilde{\mu}^{-1}_{\scriptscriptstyle\mathcal{I}_{1}}\{0\}/
\mathcal{G}^{2,p}$ is a K\"ahler metric.

\begin{theorem}
Let $\Sigma$ be a connected, compact, oriented, Riemannian surface and let $\mu'_{\scriptscriptstyle\mathcal{I}_{1}}: \mathcal{C}' \longrightarrow \mathbb{R}$ denote the restriction of the moment map $\tilde{\mu}_{\scriptscriptstyle \mathcal{I}_{1}}$ as in \eqref{moment map on config space mu 1} for the action of the gauge group $\mathcal{G}^{2,p}$ on $\mathcal{C}'$. Then the metric induced on the quotient $(\mu'_{\scriptscriptstyle\mathcal{I}_{1}})^{-1}(0)/\mathcal{G}^{2,p}:= \mathcal{M}$ is a K\"ahler metric.
\end{theorem}

\begin{proof}
The submanifold $(\mu'_{\scriptscriptstyle\mathcal{I}_{1}})^{-1}(0) \subset \mathcal{C}' \subset \mathcal{C}^{1,p}$ carries a natural $L^{2}$-metric, induced from the $L^{2}$-metric on $\mathcal{C}^{1,p}$. The action of the gauge group is by isometries, which implies that there exists a unique Riemannian metric on the quotient such that $\pi:(\mu'_{\scriptscriptstyle\mathcal{I}_{1}})^{-1}(0) \longrightarrow \mathcal{M}$ is a Riemannian submersion. Let $X, Y \in \Gamma(\mathcal{M}, T\mathcal{M})$ and $\widetilde{X}, \widetilde{Y}$ denote the horizontal lifts to $(\mu'_{\scriptscriptstyle\mathcal{I}_{1}})^{-1}(0)$. Then the covariant derivative of $\widetilde{Y}$ with respect to $\widetilde{X}$ is given by \[ \nabla^{\scriptscriptstyle \mathcal{M}}_{X}Y = \pi_{\ast} \left( \nabla^{ (\mu'_{\scriptscriptstyle \mathcal{I}_{1}})^{-1}(0)}_{\widetilde{X}} \widetilde{Y} \right), \] where, $\nabla^{ (\mu'_{\scriptscriptstyle \mathcal{I}_{1}})^{-1}(0)}$ denotes the restriction of the Levi-Civita connection on the configuration space, to $(\mu'_{\scriptscriptstyle \mathcal{I}_{1}})^{-1}(0)$. We identify the pull-back of $T\mathcal{M}$ with the horizontal sub-bundle of $T\left((\mu'_{\scriptscriptstyle \mathcal{I}_{1}})^{-1}(0)\right)$
\[\pi^{\ast}(TM) \cong \mathcal{H}\left( (\mu'_{\scriptscriptstyle \mathcal{I}_{1}})^{-1}(0) \right) = \mathcal{H}(\mathcal{C}')|_{(\mu'_{\scriptscriptstyle \mathcal{I}_{1}})^{-1}(0)}~ \bigcap~ T\left((\mu'_{\scriptscriptstyle \mathcal{I}_{1}})^{-1}(0) \right).\]
But the restriction to $(\mu'_{\scriptscriptstyle \mathcal{I}_{1}})^{-1}(0)$ of $T\mathcal{C}'$ splits $L^{2}$-orthogonally as
\begin{align*}
T\mathcal{C}'|_{(\mu'_{\scriptscriptstyle \mathcal{I}_{1}})^{-1}(0)} 
&= \mathcal{H}\left( (\mu'_{\scriptscriptstyle \mathcal{I}_{1}})^{-1}(0) \right) \oplus \left(\mathcal{H}\left( (\mu'_{\scriptscriptstyle \mathcal{I}_{1}})^{-1}(0)\right) \right)^{\perp}\\
&\cong \pi^{\ast}T\mathcal{M} \oplus \left(\mathcal{H}\left( (\mu'_{\scriptscriptstyle \mathcal{I}_{1}})^{-1}(0)\right) \right)^{\perp}\\
&= \pi^{\ast}T\mathcal{M} \oplus \text{im}(T_{0}) \oplus \left(\text{ker}(\mu'_{\scriptscriptstyle \mathcal{I}_{1}})^{\perp} \right)
\end{align*}
where $T_{0}$ denotes the linearization of the orbit map. In order to define the complex structure on $T\mathcal{M}$, it now suffices to show that $\left(\pi^{\ast}T\mathcal{M}\right)^{\perp}$ is preserved by the induced complex structure $\mathcal{I}^{'}_{1}$:
\begin{itemize}
\item Note that the image of an element $\xi \in \Omega^{0}(\Sigma, \mathbb{R})$ under the linearization of the orbit map through $q=(\mathsf{a}, u, \Phi)$ is the same as the fundamental vector field $L_{q} \xi$. We have \[\langle \mathcal{I}_{1}L_{q} \xi,~ Z\rangle = \Omega_{1}(L_{q} \xi, Z) = \langle d\mu'_{\scriptscriptstyle \mathcal{I}_{1}}(Z), \xi \rangle.\] This implies that $\mathcal{I}_{1}L_{q} \xi \perp \text{ker}(d\mu'_{\scriptscriptstyle \mathcal{I}_{1}})$.
\item Let $Z\in \text{ker}(d\mu'_{\scriptscriptstyle \mathcal{I}_{1}})$ and $\mathcal{I}_{1}Z \perp \text{im}(T_{0})$. Then for $\xi \in \Omega^{0}(\Sigma, \mathbb{R})$,
\[0 = \langle \mathcal{I}_{1}Z,~ L_{q} \xi\rangle  = -\Omega_{1}(L_{q}\xi, Z) = - \langle d\mu'_{\scriptscriptstyle \mathcal{I}_{1}}(Z), L_{q}\xi \rangle .\] Therefore $Z \in (\text{ker}(d\mu'_{\scriptscriptstyle \mathcal{I}_{1}})) \bigcap \text{ker}(d\mu'_{\scriptscriptstyle \mathcal{I}_{1}})^{\perp} = \{0\}$. 
\end{itemize}
Hence the complex structure preserves the splitting and defines a complex structure $\mathcal{I}^{\scriptscriptstyle \mathcal{M}}$ on $\mathcal{M}$.

It only remains to show that the complex structure is parallel. But this follows directly from the fact that 
\begin{itemize}
\item The complex structure $\mathcal{I}'_{1}$ on $\mathcal{C}'$ is parallel.
\item  The Levi-Civita connection on $\mathcal{M}$ is given by the projection of the Levi-Civita connection on $(\mu'_{\scriptscriptstyle \mathcal{I}_{1}})^{-1}(0)$.
\item Projection commutes with the complex structures $\mathcal{I}'_{1}$ and $\mathcal{I}^{\scriptscriptstyle \mathcal{M}}$.
\end{itemize}
Thus we have proved that the induced complex structure on $\mathcal{M}$ is parallel with repsect to the Levi-Civita connection for the induced $L^{2}$-metric on $\mathcal{M}$. Therefore the $L^{2}$-metric on $\mathcal{M}$ is K\"ahler.
\end{proof}

\section{Pre-quantum Line-bundle on moduli space}
\label{quantization}

Let $\rho$ denote the local K\"{a}hler potential for the first symplectic form $ \omega_1$ of $M$, our target hyperK\"{a}hler manifold. Local potentials exist for any K\"{a}hler form.   
Let $p \in P$. Then $u(p) \in M$. Let $V_{p}$ be a neighbourhood of $u(p)$ such that $\rho(u(p))$ is local a K\"{a}hler potential for $\omega_1$ in $V_{p}$. $u^{-1}(V_{p})$ is a covering  of $P$ which has a finite covering, namely $u^{-1}(V_{p_i})$, $i =1,2..., N$. Let  $\phi_i$ , $i = 1,...,N$ be a partition of unity subordinate to this finite covering of $P$. Let $\rho_i$, $i = 1,...,N$ be the local K\"{a}hler potential for $V_{p_i}$ for the form $\omega_1$. Let $z = \pi(p)$ be a point on the Riemann surface $\Sigma$. Define
$$\rho_0(u) = \int_{\Sigma} \Sigma_{i=1}^N \rho_i (u(\pi^{-1}(z))) \phi_i(\pi^{-1}(z)) \omega_{\Sigma}. $$

Let us define on the configuration space parametrized by the triple $(\mathsf{a}, u , \Phi)$ a Quillen determinant bundle ${\mathcal Q} = det(\partial_{\mathsf{a}})$, ~\cite{Q}, i.e. a line bundle whose the fiber over $(\mathsf{a}, u, \Phi)$ is given by \[\wedge^{top} (Ker~ \partial_{\mathsf{a}})^* \otimes \wedge^{top} (Coker ~\partial_{\mathsf{a}}).\] 

Following the idea in ~\cite{BR}, we modify the Quillen metric $exp(- \zeta_A^{\prime} (0))$ by multiplying it with $ exp( \frac{i}{4 \pi}(\rho_0(u))$ and $exp \left(\frac{i}{8 \pi} \displaystyle \int_{\Sigma} \Phi \wedge *\Phi \right)$, where $\Phi = \phi dz - \bar{\phi} d \bar{z}$. From the metric, one can calculate  the curvature  by the formula $\delta_w \delta_{\bar{w}} log || \sigma || $ where $w$ is the holomorphic coordinate on the configuation space and $\sigma $ is the canonical section of the determinant bundle \cite{Q}. The holomorphic coordinates on the configuation space  is given by $(\mathsf{a}^{0,1}, u, \Phi)$ w.r.t. the complex structure $\mathcal{I}_1$. The first term in the metric, namely  $exp(- \zeta_{\mathsf{a}}^{\prime} (0))$ contributes to the curvature by a term 
$\displaystyle\frac{i}{2 \pi} \left(\frac{-1}{2} \int_{\Sigma} \pi_{!}  \left(\alpha_1 \wedge \alpha_2 \right) \right)$, ~\cite{Q, D1}, which is the first term in $\Omega_{1}(X, Y)$. 
The second term in the metric contributes to the second term in $\Omega_1$. This can be seen as follows.
Let $u(p) = (u_1(p),..., u_n(p))$, in some local coordinate centered  at $ p \in M$ where $n = dim M$. Once again $ z = \pi(p)$. 
\begin{eqnarray*}
\gamma &=& \delta_u \delta_{\bar{u}} \rho_0(u)\\
&=& \int_{\Sigma} \Sigma_{i=1}^N 
\delta_u \delta_{\bar{u}} \rho_i(u(\pi^{-1}(z))) \phi_i(\pi^{-1}(z))) \omega_\Sigma\\
&=& \int_{\Sigma} \Sigma_{i=1}^N \Sigma_{j = 1}^{dim M}
\delta_{u_j}  \delta_{\bar{u_j}} \rho_i(u(\pi^{-1}(z))) \phi_i(\pi^{-1}(z))) \omega_\Sigma\\
&=& \frac{1}{2}\int_{\Sigma} \Sigma_{i=1}^N g^{\scriptscriptstyle M}_{u} \left(I_{1} \cdot, \cdot \right) \phi_i(\pi^{-1}(z))  \omega_{\scriptscriptstyle\Sigma}\\
&=& \frac{1}{2}\int_{\Sigma}  g^{\scriptscriptstyle M}_{u} \left(I_{1} \cdot, \cdot \right)   \omega_{\scriptscriptstyle\Sigma}
\end{eqnarray*} 
where we have used the fact that since $\rho_i$ is a K\"{a}hler potential for $\omega_1$ on $V_{p_i}$, $\Sigma_{j = 1}^{dim M}
\delta_{u_j}  \delta_{\bar{u_j}} \rho_i(u(\pi^{-1}(z))) = g^{\scriptscriptstyle M}_{u} \left(I_{1} \cdot, \cdot \right)$  and $\Sigma_{i=1}^N \phi_i(\pi^{-1}(z)) =1$.  Then $\gamma ( \xi_1, \xi_2) =  \displaystyle \int_{\Sigma} g^{\scriptscriptstyle M}_{u} ( I_{1} \xi_1, \xi_2) \omega_{\scriptscriptstyle \Sigma}$. The third term in the metric contributes to the third term in $\Omega_1$.
This can be seen as follows: Recall $\Phi = \phi dz - \bar{\phi} d \bar{z}$,~ $*\Phi =  \bar{\phi} d \bar{z} + \phi d z$ 
so that $\displaystyle exp\left(\frac{i}{8 \pi}  \int_{\Sigma} \Phi \wedge *\Phi \right) = \displaystyle exp \left(\frac{i}{4 \pi}  \int_{\Sigma} (\phi \bar{\phi}) dz \wedge d \bar{z}\right)$.
Let
\begin{eqnarray*}
\tau &=&\delta_{\phi} \delta_{ \bar{\phi}} \left(log  \left(exp\left(\frac{i}{4 \pi} \displaystyle \int_{\Sigma} (\phi \bar{\phi}) dz \wedge d \bar{z}\right)\right)\right) \\
&=& \frac{i}{4 \pi}  \int_{\Sigma}\delta_{\phi} \delta_{\bar{\phi}}( \phi  \bar{\phi}) dz \wedge d \bar{z} \\
&=& \frac{i}{4 \pi} \int_{\Sigma} \left(\delta \phi \otimes \delta \bar{\phi} - \delta \bar{\phi} \otimes \delta \phi \right) d z \wedge d \bar{z}
\end{eqnarray*}
Then, $\displaystyle \tau (\eta_1, \eta_2) = \frac{i}{4 \pi} \int_{\Sigma} \eta_1 \wedge \eta_2$. The three terms combined gives us the following proposition:

\begin{proposition}\label{quillen construction}
On the configuration space, the Quillen bundle ${\mathcal Q}$ equipped with the modified metric mentioned 
above has curvature $\frac{i}{2 \pi} \Omega_1$.
\end{proposition}

As in \cite{D1, D2}, it can be shown that this line bundle descends to the moduli space as long as the descendent of $\Omega_1$ is integral. It  is holomorphic and is a prequantum bundle since its curvature is proportional to the symplectic form with the proportionality constant $\frac{i}{2 \pi}$. One can take holomorphic square integrable sections of this bundle as the Hilbert space of quantization.

\section{Summary and discussion}

The dimensional reduction technique mentioned gives us a generalization of the Symplectic vortex equations ($\Phi = 0$). It is well-known that the invariants for Hamiltonian group actions on a symplectic manifold are related to Gromov-Witten invariants for its symplectic reduction \cite{ziltner}. Assuming that we atleast have the moduli space of finite volume, for $\Phi = 0$, we must get equivalence between invariants for Hamiltonian group action on $\mu_c \circ u = 0$ and Gromov-Witten invariants for hyperK\"ahler reduction of $M$. The latter are known to be trivial! This gives rise to an interesting question as to whether the presence of a non-zero Higgs-field helps us define non-trivial, Gromov-Witten like invariants for hyperK\"ahler manifolds. 


\bibliographystyle{plain}

\bibliography{Generalized_Seiberg_Witten_on_Riemann_surfaces_Varun_Thakre}

\end{document}